\documentclass{amsart}
\author{Ofir Gorodetsky}
\dedicatory{To Professor Helmut Maier on his 70th birthday} 
\address{Department of Mathematics, Technion -- Israel Institute of Technology, Haifa, Israel}
\email{ofir.gor@technion.ac.il}
\title{Sharp local estimates for smooth numbers}
\thanks{This project has received funding from the European Research Council (ERC) under the European Union's Horizon 2020 research and innovation programme (grant agreement no.~851318) and from the Israel Science Foundation (grant no.~2088/24). The author is also supported by the Rabbi Dr. Roger Herst Faculty Fellowship and an Alon Fellowship.}
\usepackage[margin=1in]{geometry}
\usepackage{amssymb}
\usepackage{amsfonts}
\usepackage{amsthm}
\usepackage{amsmath}
\usepackage{hyperref}
\usepackage{mathtools}

\theoremstyle{plain}
\newtheorem{thm}{Theorem}[section]
\newtheorem{lem}[thm]{Lemma}  
\newtheorem{prop}[thm]{Proposition}
\newtheorem{cor}[thm]{Corollary}

\theoremstyle{remark}
\newtheorem{rem}{Remark}

\newcommand{\PP}{\mathbb{P}}
\newcommand{\EE}{\mathbb{E}}
\newcommand{\ZZ}{\mathbb{Z}}
\newcommand*\diff{\mathop{}\!\mathrm{d}}
\newcommand{\ubar}{\bar{u}}
\newcommand{\additive}{\Delta}

\begin{document}

\begin{abstract}
	We establish sharp local estimates for smooth numbers. Corollaries are derived, including optimal estimates for square-free smooth numbers.
\end{abstract}

\maketitle

\section{Local estimates}
A positive integer is called $y$-smooth if its prime factors do not exceed $y$. We denote by $S(x,y)$ the set of $y$-smooth numbers up to $x$ and let $\Psi(x,y)=\#S(x,y)$. Throughout, $x \ge y \ge 2$. Let $\alpha=\alpha(x,y)$ be the unique positive solution to $\sum_{p \le y}\log p/(p^{\alpha}-1)=\log x$, which satisfies \cite[Thm.~2]{HildebrandTenenbaum1986}
\begin{equation}\label{eq:alphasize} 	\alpha = \frac{\log(1+y/\log x)}{\log y} \left(1+O\left(\frac{\log \log (2y)}{\log y}\right)\right).
\end{equation}
In this section, we state two new local estimates for $\Psi(x,y)$, which will be proven in \S\ref{sec:technical}. These are estimates for the ratio $\Psi(x/d,y)/\Psi(x,y)$, usually in terms of $\alpha$; see \cite{Hensley1986,Hildebrand1986local,HildebrandTenenbaum1986,Ivic,DLBT20052,DeLa2017} for previous works. As our estimates are quite technical, we begin by stating a usable corollary.
\begin{cor}\label{cor:usable}
For $2 \le d \le y$ and $x \ge y \ge 2$,
\begin{equation}\label{eq:state2more4} \Psi(x/d,y)=\frac{\Psi(x,y)}{d^{\alpha}} \left(1+O\left( \frac{\log d}{\log x} \right)\right).
\end{equation}
\end{cor}
This corollary is a special case of Corollary~\ref{cor:app} proved in \S\ref{sec:overview}. Let $N_{x,y}$ be an integer chosen uniformly at random from $S(x,y)$.  Corollary~\ref{cor:usable} has the following implication:
\begin{equation}\label{eq:imp}
\PP(d \text{ divides }N_{x,y}) =\frac{1}{d^{\alpha}}\left( 1+ O\left(\frac{\log d}{\log x}\right)\right)
\end{equation}
holds for $d \in [1,y]\cap \ZZ$. When $d=2$, say, Corollary~\ref{cor:usable} is already new, except in two extreme scenarios: when $\log y \asymp \log x$ \cite[Thm.~2.4]{DLBT20052} or when $y\ll (\log x \log^2 \log x)^{1/3}$ \cite[Cor.~3.2]{DeLa2017}. 
Let\[\Psi_{\mu^2}(x,y):=\sum_{n \in S(x,y)} \mu^2(n)\] be the number of square-free $y$-smooth integers up to $x$. Let
\[ \zeta(s,y) := \prod_{p \le y}(1-p^{-s})^{-1}.\]
The following corollary is an application of our estimates, and its proof is given in \S\ref{sec:sfree}.
\begin{cor}\label{cor:sfinto}
	Fix $\varepsilon>0$. Uniformly for $x \ge y \ge (\log x)^{2+\varepsilon}$, we have
	\[ 	\frac{\Psi_{\mu^2}(x,y)}{\Psi(x,y)}=\frac{1}{\zeta(2\alpha,y)}\left(1+ O_{\varepsilon}\left(\frac{1}{\log x}\right)\right).\]
\end{cor}
Corollary~\ref{cor:sfinto} improves on estimates in \cite{Ivic1985,Ivic,Naimi,DLBT20052}. In Proposition \ref{prop:sfree}, we investigate $\Psi_{\mu^2}(x,y)/\Psi(x,y)$ further, obtaining
\begin{cor}\label{cor:sfreerange}
We have 
\begin{equation}\label{eq:Psiratio}
	\frac{\Psi_{\mu^2}(x,y)}{\Psi(x,y)} \sim \frac{1}{\zeta(2\alpha,y)}
\end{equation}
when $y/( (\log x)^{3/2}/\sqrt{\log \log x})\to \infty$. Relation \eqref{eq:Psiratio} does not hold for $y \asymp (\log x)^{3/2}/\sqrt{\log \log x}$.
\end{cor}
Let \[u := \frac{\log x}{\log y} \quad \text{and} \quad \ubar:= \min\left\{u, \frac{y}{\log y}\right\}.\]
We state our two technical estimates.  
\begin{lem}\label{lem:ularge}
	Fix $\varepsilon \in (0,3/2)$.
	Uniformly for $0<d\le x$ and $x \ge y \ge 2$, we have 
	\[ \Psi(x/d,y) = \frac{\Psi(x,y)}{ d^{\alpha} }\bigg( 1+ O_{\varepsilon}\bigg( \frac{|\log d|}{\log x} \left( \frac{\log y}{y} + \frac{1}{\max\{1,(\log x)/y\}}\right)+\frac{\log^2 d}{(\log x)(\log y) \max\{1,(\log x)/y\}}+ \additive\bigg)\bigg)\]
	where 
	\begin{equation}\label{eq:delta}
		\additive:=\sqrt{(\log x )(\log y) \max\{1,(\log x)/y\}} \bigg(\frac{\exp(-c\ubar^{1/3})}{u}  +\alpha \exp(-\min\{ cu\log(2u)^{-2},(\log y)^{3/2-\varepsilon}\})\bigg).
	\end{equation}
	Furthermore, in the same range, $\Psi(x/d,y) \ll_{\varepsilon} \Psi(x,y)d^{-\alpha}(1+\additive)$.
\end{lem}
\begin{lem}\label{lem:smallu}
Uniformly for $1 \le d \le x$ and $x \ge y \ge  \exp(\sqrt{\log x} \log \log x)$, we have
	\begin{equation}\label{eq:state2more2} \Psi(x/d,y)=\frac{\Psi(x,y)}{d^{\alpha}} \left(1+O\left( \frac{1+\log d}{\log x} + \min\left\{1, \frac{\log^2 d}{(\log x)(\log y)}\right\}\right)\right).
	\end{equation}
\end{lem}
See Remark \ref{rem:smallu} for an improvement on Lemma \ref{lem:smallu}.

The rest of the section consists of applications of Lemmas~\ref{lem:ularge}--\ref{lem:smallu}. Throughout the paper, $C$ and $c$ denote absolute positive constants that may change from one occurrence to another.
\subsection{Consequences}\label{sec:overview}
\begin{cor}\label{cor:simpcor}
	Fix $\varepsilon >0$. Uniformly for $1 \le d \le x$ and $x \ge y \ge \exp((\log \log x)^{2/3+\varepsilon})$, we have
\[\Psi(x/d,y)	=	\frac{\Psi(x,y)}{d^{\alpha}} \bigg(1+O_{\varepsilon}\bigg( \frac{1+\log d}{\log x}\left( \frac{\log y}{y} + \frac{1}{\max\{1,(\log x)/y\}}\right) + \frac{\log^2 d}{(\log x)( \log y) \max\{1,(\log x)/y\}}\bigg) \bigg).\]
\end{cor}
\begin{proof}
	This is in Lemma \ref{lem:smallu} if $x \ge y \ge \exp(c\log x/(\log \log x)^3)$ and in Lemma \ref{lem:ularge} otherwise.
\end{proof}
We review previous works and use them to further improve Corollary \ref{cor:simpcor}. Theorem 2.4(i) of La Bret\`eche and Tenenbaum \cite{DLBT20052} applied with $m=1$ says $\Psi(x/d,y) \ll \Psi(x,y)/d^{\alpha}$ holds for $1 \le d \le x$ and $2 \le y \le x$. This bound also follows from the last part of Lemma \ref{lem:ularge} (if $y \le \exp(c\log x/(\log \log x)^3)$) and Lemma \ref{lem:smallu} (for larger $y$). Set 
\[	u_y :=  \ubar+ \frac{\log y}{\log(u+2)}. \]
Theorem 2.4(ii) of La Bret\`eche and Tenenbaum \cite{DLBT20052} applied with $m=1$ says that, uniformly for $x \ge y \ge 2$ and $1 \le d \le x$,
\begin{align}\label{eq:state}
	\Psi(x/d,y)&=\frac{\Psi(x,y)}{d^{\alpha}} \left(1+O\left(\frac{1}{u_y}+\frac{\log d}{\log x}\right)\right) \left(1-\frac{\log^2 d/\log^2 y}{u^2+\ubar^2}\right)^{b\ubar}\\
	\label{eq:state2}	&=\frac{\Psi(x,y)}{d^{\alpha}} \left(1+O\left(\frac{1}{u_y}+\frac{\log d}{\log x} + \frac{\log^2 d}{(\log x)( \log y) \max\{1,(\log x)/y\}}\right)\right)
\end{align}
Here $b=b(x,y,d)$ is an unspecified function taking values in $[c,C]$. The equality in \eqref{eq:state2} follows from \eqref{eq:state} and the inequality $0 \le 1-(1-a)^b  \ll ab$ for $0 \le a < 1$ and $b\gg 1$.\footnote{If $b\ge 1$ and $0 \le a <1$, $1-(1-a)^b\le ab$ is Bernoulli's inequality. If $0<b\le 1$ and $0\le a<1$, we have $1-(1-a)^b \le a$ by monotonicity.} The following corollary sharpens \eqref{eq:state} when $x \ge y \ge \log x$.
\begin{cor}\label{cor:appsmalll}
	Uniformly for $1 \le d \le x$ and $x \ge y \ge \log x$, we have
\[\Psi(x/d,y)=\frac{\Psi(x,y)}{d^{\alpha}} \left(1+O\left( \frac{1+\log d}{\log x}\right)\right)\left(1-\frac{\log^2 d}{2\log^2 x}\right)^{bu}.\]
\end{cor}
\begin{proof}
	If $d \le y$, this follows from Corollary \ref{cor:simpcor}. If $d>y$, this follows from \eqref{eq:state}. 
\end{proof}
\begin{rem}
We may replace $(1-\log^2 d/(2\log^2 x))^{bu}$ in Corollary \ref{cor:appsmalll} with $\exp(-\tilde{b}\log^2 d /(\log x \log y))$ for some function $\tilde{b}$ taking values in $[c,C]$. 
\end{rem}
In \cite[Cor.~3.2; $m=1$]{DeLa2017}, La Bret\`eche and Tenenbaum showed that if $y \le (\log x)^{1-\varepsilon}$ and $d \le x^{1-\varepsilon}$, then
\begin{align}
\nonumber	\Psi(x/d,y)&=\frac{\Psi(x,y)}{d^{\alpha}} \left(1+O_{\varepsilon}\left(\frac{1}{\exp((\log y)^{3/2-\varepsilon})\log x}+\frac{y^3}{\log^2 x(\log y) }\right)\right)\left(1-\frac{\log^2 d/\log^2 y}{u^2+\ubar^2}\right)^{b_{\varepsilon}\ubar}\\
	\label{eq:state2more22}		& =\frac{\Psi(x,y)}{d^{\alpha}} \left(1+O_{\varepsilon}\left(\frac{1}{\exp((\log y)^{3/2-\varepsilon})\log x} + \frac{y^3}{\log^2 x (\log y)} + \frac{\log^2 d}{\log^2 x (\log y)/y}\right)\right).
\end{align}
Here $\varepsilon >0$ is fixed and $b_{\varepsilon}=b_{\varepsilon}(x,y,d)$ is an unspecified function taking values in $[c_{\varepsilon},C_{\varepsilon}]$. We deduce
\begin{cor}\label{cor:app}
	Uniformly for $1 \le d \le x$ and $x \ge y \ge 2$, we have
\[	\Psi(x/d,y)	=	\frac{\Psi(x,y)}{d^{\alpha}} \left(1+O\left( \frac{1+\log d}{\log x}\left( \frac{\log y}{y} + \frac{1}{\max\{1,(\log x)/y\}}\right) + \frac{\log^2 d}{(\log x)( \log y) \max\{1,(\log x)/y\}}\right) \right). \]
\end{cor}
\begin{proof}
	This follows from Corollary \ref{cor:simpcor} when $y \ge \exp((\log \log x)^{3/4})$, from \eqref{eq:state2more22} if $y <  \exp((\log \log x)^{3/4})$ and $d \le x^{2/3}$, and from $\Psi(x/d,y) \ll \Psi(x,y)/d^{\alpha}$ when $x \ge d > x^{2/3}$.
\end{proof}
\begin{rem}
Corollaries \ref{cor:appsmalll}--\ref{cor:app} can be improved when $d \sim 1$.
\end{rem}
\subsection{Proof of Corollary~\ref{cor:sfinto}: square-free smooth numbers}\label{sec:sfree}
When $x\ge y \ge (\log x)^{2+\varepsilon}$, La Bret\`eche and Tenenbaum \cite[Cor.~2.6]{DLBT20052} established that
\begin{equation}\label{eq:oldmu}
	\frac{\Psi_{\mu^2}(x,y)}{\Psi(x,y)}=\frac{1}{\zeta(2\alpha,y)} \left(1+ O_{\varepsilon}\left(\frac{1}{u_y}\right)\right).
\end{equation} 
The original proof of \eqref{eq:oldmu} uses the identity $\mu^2(n)=\sum_{d^2 \mid n} \mu(d)$ in the form
	\[ \sum_{n \in S(x,y)} \mu^2(n) = \sum_{d \in S(\sqrt{x},y)}\mu(d) \Psi(x/d^2,y)\] 
	and then applies \eqref{eq:state2} to each term (except $d=1$, for which $\Psi(x/d^2,y)=\Psi(x,y)$).
	In the range $y \ge (\log x)^2$ and $d \ge 2$, Corollary \ref{cor:simpcor} allows us to drop the term $1/u_y$ in \eqref{eq:state2}. The term $1/u_y$ is the most restrictive term, and its removal leads to an improved error by following the  proof of \eqref{eq:oldmu}. In this way, we obtain Corollary~\ref{cor:sfinto}.
\subsection{General sums}
In \cite{DLBT20052}, La Bret\`eche and Tenenbaum proved various results on sums of the shape
\[ \Psi(x,y;f):=\sum_{n \in S(x,y)} f(n) \]
where $f$ is a given function (defined on a set containing the positive integers). We describe those results. Defining the arithmetic function $h=f*\mu$, they let
\[ M_j = M_j(x,y;f):= \sum_{d \in S(x,y)} \frac{|h(d)|(\log d)^j}{d^{\alpha}}\]
for $j=0,1,2$. In \cite[Thm.~2.5]{DLBT20052}, they prove that
\[ \frac{\Psi(x,y;f)}{\Psi(x,y)}=\sum_{d \in S(x,y)} \frac{h(d)}{d^{\alpha}} + O\left(\frac{M_0}{u_y}+\frac{M_1}{\log x}+\frac{M_2 \ubar}{(\log x)^2}\right)\]
holds uniformly for $x \ge y \ge 2$. 
If $f$ is a continuously differentiable function defined on  $[1,\infty)$, they let
\[ M_j^*=M_j^*(x,y;f):=\int_{1}^{x} \frac{|f'(t)|(\log (x/t))^j}{(x/t)^{\alpha}}\diff{t}\]
for $j=0,1,2$. In \cite[Thm.~2.7]{DLBT20052}, they prove that
\begin{equation}\label{eq:thm27}
	\frac{\Psi(x,y;f)}{\Psi(x,y)}= f(x)-\int_{1}^{x} \frac{f'(t)}{(x/t)^{\alpha}}\diff{t} +O\left(\frac{M_0^*}{u_y}+\frac{M_1^*}{\log x}+\frac{M_2^* \ubar}{(\log x)^2}\right)
\end{equation}
holds uniformly for $x \ge y \ge 2$. If $f$ is an additive arithmetic function, they let
\[ m_j=m_j(x,y;f):=\sum_{p^{\nu}\in S(x,y)} \frac{|f(p^{\nu})|(1-p^{-\alpha})(\log (p^{\nu}))^j}{p^{\nu \alpha}}\]
for $j=0,1,2$. In \cite[Thm.~2.9]{DLBT20052}, they prove that
\[	\frac{\Psi(x,y;f)}{\Psi(x,y)}= \sum_{p^{\nu} \in S(x,y)}f(p^{\nu}) \frac{1-p^{-\alpha}}{p^{\nu\alpha}} + O\left(\frac{m_0}{u_y}+\frac{m_1}{\log x}+\frac{m_2 \ubar}{(\log x)^2}\right). \]
\begin{cor}\label{cor:M0remove}
	The error terms $M_0/u_y$ and $m_0/u_y$ in Theorems 2.5 and 2.9 of \cite{DLBT20052} may be omitted. The error term $M_0^*/u_y$ in Theorem 2.7 of \cite{DLBT20052} may be replaced with  $M_0^*/\log x $.
\end{cor}
\begin{proof}
	Theorems 2.5 and 2.9 of \cite{DLBT20052} follow quickly from \eqref{eq:state2} for $d \ge 2$; see page 147 of \cite{DLBT20052}. Our Corollary \ref{cor:app} says one may omit the term $1/u_y$ in \eqref{eq:state2} if $d \ge 2$, which is the reason for the existence of $M_0/u_y$ and $m_0/u_y$ in the first place. As for Theorem 2.7, it relies on 
	\[ \Psi(x/t,y)=\frac{\Psi(x,y)}{t^{\alpha}}\left(1+ O\left( \frac{1}{u_y} + \frac{\log t}{\log x} + \frac{ \log^2 t}{(\log x )(\log y) \max\{1,(\log x)/y\}}\right)\right)\] 
	for real $1 \le t \le x$; see page 197 of \cite{DLBT20052}.
	According to Corollary \ref{cor:app}, we may replace $1/u_y$ with $1/\log x$, leading to $M_0^*/u_y$ being improved to $M_0^*/\log x$.
\end{proof}
\begin{rem} 
When $y=o(\log x)$, one can improve the size of the terms $M_1/\log x$ and $m_1/\log x$ by utilizing the full strength of Lemma \ref{lem:ularge} and \eqref{eq:state2more22}. 
\end{rem}
We describe a special case of Corollary \ref{cor:M0remove}. In \cite[Cor.~2.8]{DLBT20052}, it is proved that 
\[	\frac{\Psi(x,y;\log)}{\Psi(x,y)} = \log x-\frac{1}{\alpha}+O\left( \frac{1}{\alpha u_y}\right) \]
holds uniformly for $x \ge y \ge 2$, as a special case of \eqref{eq:thm27}. Here $\log$ is the natural logarithm. Using Corollary \ref{cor:M0remove} with $f=\log$, we obtain, following the proof of \cite[Cor.~2.8]{DLBT20052}, the following
\begin{cor}\label{cor:log}
	Uniformly for $x \ge y \ge 2$,
	\[ \frac{\Psi(x,y;\log)}{\Psi(x,y)} = \log x-\frac{1}{\alpha}+O\left( \frac{1}{\alpha^2 \log x}\right).\]
\end{cor}
\subsection{Smooth integers coprime to \texorpdfstring{$m$}{m}}
Given a $y$-smooth integer $m$, let $\Psi_m(x,y)$ be the number of $y$-smooth integers up to $x$ that are coprime to $m$. In \cite[Thm.~2.1]{DLBT20052}, La Bret\`eche and Tenenbaum estimate $\Psi_m(x,y)$ in terms of $\Psi(x,y)$. Precisely, they bound the quantity
\begin{equation}\label{eq:gmratioquant} \frac{\Psi_m(x,y)}{\Psi(x,y)g_m(\alpha)} - 1
\end{equation}
where
\[g_m(s):=\prod_{p \mid m}(1-p^{-s}).\]
In \cite[Thm.~2.4]{DLBT20052}, they estimate $\Psi_m(x/d,y)$ in terms of $\Psi(x,y)$. They do so in two steps: first, one uses the bound for \eqref{eq:gmratioquant} to estimate $\Psi_m(x/d,y)$ in terms of $\Psi(x/d,y)$, and then  one uses the $m=1$ case of \cite[Thm.~2.4]{DLBT20052} to estimate $\Psi(x/d,y)$ in terms of $\Psi(x,y)$. Since in Corollary \ref{cor:appsmalll} we improved on the $m=1$ case of \cite[Thm.~2.4]{DLBT20052} when $y \ge \log x$, we can improve the general $m$ case of \cite[Thm.~2.4]{DLBT20052} by following the proof of \cite[Thm.~2.4]{DLBT20052} but using Corollary \ref{cor:appsmalll} in the process. One obtains in this way the following
\begin{cor}\label{cor:cop}
	For $x \ge y \ge \log x$, $1 \le d \le x/y$, and a $y$-smooth integer $m$ such that $\omega(m) \ll \sqrt{y}/(\log y)^{\delta}$,
\[	\frac{\Psi_m(x/d,y)}{\Psi(x,y)} =  \frac{g_m(\alpha)}{d^{\alpha}}\left( 1+O\left(\frac{1}{\log x}+(1+E_m)\left( \frac{\log d}{\log x} + \frac{E_m}{u}\right)\right)\right)\exp\left(-b_2 \frac{\log^2 d}{\log x \log y} \right).\]
Here $b_2=b_2(x,y,d)$ is a function taking values in $[c,C]$, and $E_m=E_m(x,y)$, $\delta=\delta(x,y)\in [0,1/2]$ are functions defined in \cite[Eqs.~(2.5) and (2.8)]{DLBT20052}.
 If $\omega(m) \ll 1$, then the range of $d$ may be extended to $1 \le d \le x$.
\end{cor}
Previously, Corollary \ref{cor:cop}  was proved with the error term $1/\log x$ replaced by $1/u_y$ \cite[Thm.~2.4]{DLBT20052}.
\section{The anatomy of smooth integers}
Many works \cite{AlladiTuran,Xuan1,Xuan2,DLBT2005,DLB2,DLB3,DLB4,DLB2016,DLB6} were devoted to the study of the Tur{\'a}n--Kubilius inequality for additive functions supported on $y$-smooth numbers. Given $t \in [2,y]$, let \[\omega_t(n):=\sum_{p \mid n, \, p \le t}1.\]
We consider the variance of $\omega_t(N_{x,y})$: \[V_t(x,y):=\EE (\omega_t(N_{x,y})-\EE\omega_t(N_{x,y}))^2.\]
Applying the general Tur{\'a}n--Kubilius inequality of La Bret\`eche, Lamzouri, and Tenenbaum \cite[Thm.~1.1]{DLB6} to the additive function $\omega_t$, one derives the following special case.
\begin{cor}[La Bret\`eche--Lamzouri--Tenenbaum]\label{cor:special}
	Suppose $2 \le t \le y \le x$. When $\min\{y,u\} \to \infty$, we have \[V_t(x,y) \le (1+o(1)) \sum_{p\le t}p^{-\alpha}(1-p^{-\alpha}).\] 
\end{cor}
Upper bounds on $V_t$ are essential to the study of the anatomy of smooth numbers \cite{DLBT2005,DLB2016}. In this section, we give a short proof of Corollary \ref{cor:special}, based on the following variant of Lemma \ref{lem:ularge}.
\begin{lem}\label{lem:psipsi}	Fix $\varepsilon>0$.	Let $x \ge y \ge 2$. Suppose $1 \le a,b$ and $ab \le x$. Then
	\begin{multline*} \frac{\Psi(x/(ab),y)}{\Psi(x,y)} -\frac{\Psi(x/a,y)\Psi(x/b,y)}{\Psi(x,y)^2} \\= -(ab)^{-\alpha}\left( \frac{\log a \log b}{\phi_2}\left(1 + O\left(\frac{1}{\ubar}+\frac{\log a \log b}{\phi_2}+ \frac{\log(ab)}{\phi_2^{1/2}}\right)\right) + O_{\varepsilon}(\additive + \additive^2)\right),
	\end{multline*}
	where $\additive$ is defined in \eqref{eq:delta}, and 
	\begin{equation}
		\label{eq:phi_2}
		\phi_2=\phi_2(\alpha,y):=\sum_{p \le y}\frac{p^{\alpha}\log^2 p}{(p^{\alpha}-1)^2}.
	\end{equation}
\end{lem}
Note that $\phi_2(s,y)$ is the second partial derivative of $\log \zeta(s,y)$ with respect to $s$. Hildebrand and Tenenbaum \cite[Thm.~2]{HildebrandTenenbaum1986}  proved that 
\begin{equation}\label{eq:phi_2_est}
\phi_2(\alpha,y)=\bigg(1+\frac{\log x}{y}\bigg)(\log x)(\log y ) \bigg(1+O\bigg(\frac{1}{\log(\ubar+1)}\bigg)\bigg)
\end{equation}
holds in $x \ge y \ge 2$. Lemma~\ref{lem:psipsi} is proved in \S\ref{sec:technical}. We proceed to prove Corollary \ref{cor:special}. Let $2\le t \le y \le x$. We write $\omega_t(n)$ as a sum of indicators $\mathbf{1}_{p \mid n}$ to find that
	$\omega_t(N_{x,y}) - \EE(\omega_t(N_{x,y}))=\sum_{p\le t} X_p$ holds for 
	\[ X_p :=  \mathbf{1}_{p \mid N_{x,y}} - \EE \mathbf{1}_{p \mid N_{x,y}} = \mathbf{1}_{p \mid N_{x,y}} - \frac{\Psi(x/p,y)}{\Psi(x,y)}.\]
	Using \eqref{eq:state2more4},
	\begin{multline*}
	\sum_{p\le t}	\EE X_p^2 = \sum_{p\le t} \frac{\Psi(x/p,y)}{\Psi(x,y)} - \left(\frac{\Psi(x/p,y)}{\Psi(x,y)}\right)^2\\
		= \sum_{p \le t} p^{-\alpha} \left( 1 + O\left(\frac{\log p}{\log x}\right)\right) \left( 1 - p^{-\alpha} + O\left( p^{-\alpha} \frac{\log p}{\log x} \right)\right)= (1+o(1))\sum_{p \le t} p^{-\alpha}(1-p^{-\alpha})
	\end{multline*}
	as $\min\{y,u\}\to \infty$. Since $V_t(x,y) = \sum_{p \le t} \EE X_p^2 + 2 \sum_{p < q \le t} \EE X_p X_q$, it remains to show $\sum_{p < q \le t} \EE X_p X_q \le o(1)\sum_{p\le t}p^{-\alpha}(1-p^{-\alpha})$ as $\min\{y,u\} \to \infty$. The rest of the proof is divided into three cases.
	\subsection{Range 1: $(\log \log x)^{3/4} \le \log y \le c\log x/(\log \log x)^3$}
	In this range, by Lemma \ref{lem:psipsi}, we have
		\[ \frac{\Psi(x/(pq),y)}{\Psi(x,y)} -\frac{\Psi(x/p,y)\Psi(x/q,y)}{\Psi(x,y)^2} = -\left( 1+ O\left(\frac{1}{\ubar} + \frac{\log (pq)}{\phi_2^{1/2}} \right)\right) \frac{\log p \log q}{ (pq)^{\alpha}\phi_2}< 0\]
		when $p <q \le y$ and $x$ is sufficiently large. Hence, $\EE X_p X_q < 0$, concluding the proof.
		\subsection{Range 2: $\log y \le (\log \log x)^{3/4}$}
		In this range,
		$\EE X_p X_q \ll \exp(-(\log y)^{5/4})/((pq)^{\alpha}\log x)$ holds by \eqref{eq:state2more22} for $p<q \le y$.  Hence,
		\[ \sum_{p <q \le t} \EE X_p X_q \ll  \frac{( \sum_{p \le t} p^{-\alpha})^2}{ \exp((\log y)^{5/4})\log x}. \]
		By \eqref{eq:alphasize}, $\alpha \asymp y/(\log x \log y)$ holds when $y  \le \log x$, implying $\sum_{p \le t}p^{-\alpha} \asymp t/\log t$ and $\sum_{p \le t}p^{-\alpha}(1-p^{-\alpha}) \asymp \alpha t$ when $t \le y \le \log x$. In particular,
		\[ \sum_{p <q \le t} \EE X_p X_q \ll \frac{1}{\exp((\log y)^{5/4})} \sum_{p \le t}p^{-\alpha}(1-p^{-\alpha}),  \]
which is sufficiently small if $y\to \infty$.
		\subsection{Range 3: $\log y \ge c \log x/(\log \log x)^3$ and $u \ge C$}
		In this range, we argue as in Alladi \cite{AlladiTuran}. Alladi proved that $\log \rho(t)$ is concave once $t$ is large enough \cite[Lem.~4]{AlladiTuran}. Recall de Bruijn showed $\Psi(x,y) = x\rho(u)(1+O(\log (u+1)/\log y))$ holds in the considered range \cite{db}. Given a prime $p$, let $a_p := \log p /\log y$. If $p<q\le y$ and $u$ is sufficiently large then de Bruijn's result shows that in our range, $\EE X_p X_q$ is
		\begin{align*}
			&= \frac{1}{pq} \frac{\rho(u-a_p-a_q)}{\rho(u)}\left(1+O\left(\frac{\log(u+1)}{\log y}\right)\right) - \frac{1}{pq}\frac{\rho(u-a_p)\rho(u-a_q)}{\rho(u)^2}\left(1+O\left(\frac{\log(u+1)}{\log y}\right)\right) \\
			&\le C \frac{\log(u+1)}{\log y}\frac{1}{pq}  \frac{\rho(u-a_p)\rho(u-a_q)}{\rho(u)^2}
		\end{align*}
		by using  concavity of $\log \rho$. Define the function $\xi\ge 0$ via $e^{\xi(u)}=1+u\xi(u)$. Since $\rho(u-t)/\rho(u) \ll e^{t\xi(u)}$ \cite[Cor.~2.4]{HT1993}, it follows that
		\[ \sum_{p<q \le t}\EE X_p X_q \le C \frac{\log(u+1)}{\log y} \big(\sum_{p \le t} p^{-\beta}\big)^2\]
		for
		\begin{equation}\label{eq:betadef}
			\beta=\beta(x,y):=1-\frac{\xi(u)}{\log y}.
		\end{equation}
		It is known that \cite[Eq.~(7.8)]{HildebrandTenenbaum1986}
		\begin{equation}\label{eq:betaalpha}
			\beta= \alpha +O_{\varepsilon}\left(\frac{1}{(\log x)( \log y)} + \exp(-(\log y)^{3/5-\varepsilon})\right)
		\end{equation}
		holds uniformly for $x \ge y \ge (\log x)^{1+\varepsilon}$. To conclude, we observe that $\sum_{p\le t}p^{-\beta} \sim \sum_{p \le t}p^{-\alpha} \asymp \sum_{p\le t} p^{-\alpha}(1-p^{-\alpha})$ and $\sum_{p \le y}p^{-\alpha} \ll \log \log y + u$ hold in our range using \eqref{eq:betaalpha}
		and \cite[Lem.~3.6]{DLBT20052}.
		\begin{rem}
			Suppose $y=o(\sqrt{\log x \log \log x})$. The probability that $p \nmid N_{x,y}$ is $\ll \alpha \log p + \log p/ \log x$ by \eqref{eq:imp}. These probabilities sum to $o(1)$ as we range over $p\le y$. Thus, almost all elements of $S(x,y)$ are divisible by $\prod_{p \le y}p$.
		\end{rem}
\section{More on square-free smooth numbers}\label{sec:moresfree}
Let $\vartheta(y):=\sum_{p \le y} \log p \sim y$. For $\log x\ge \vartheta(y)$, we have $\Psi_{\mu^2}(x,y)=2^{\pi(y)}$ because a $y$-smooth square-free integer is of size $\le \prod_{p \le y}p$. In view of Corollary \ref{cor:sfinto}, we study $\Psi_{\mu^2}(x,y)/\Psi(x,y)$ only for $y$ satisfying $y \le (\log x)^3$ and $\vartheta(y)>\log x$.

De Bruijn \cite{debruijn1966} (resp.~Granville \cite{Granville}) estimated $\log \Psi(x,y)$ (resp.~$\log \Psi_{\mu^2}(x,y)$) asymptotically as $y \to \infty$. Let $v:=y/\log x$ and fix $A>0$. Subtracting their estimates, we find that if $v \in [2,A]$, then
\[\log \frac{\Psi_{\mu^2}(x,y)}{\Psi(x,y)} \sim \frac{\log x}{\log \log x}	\left( (v-1)\log\left(1 + \frac{1}{v-1}\right) - (v+1) \log \left(1+\frac{1}{v}\right)\right)\]
holds as $y \to \infty$ and  if $v \in [A^{-1},2]$, then
\[\log \frac{\Psi_{\mu^2}(x,y)}{\Psi(x,y)} \sim -\frac{\log x}{\log \log x}	\left( \log(1+v)+v\left(\log\left(1+\frac{1}{v}\right)-\log 2\right)\right)\]
holds as $y \to \infty$.  From now on, we assume further $y\ge 3\log x$.

Naimi \cite{Naimi} studied $\Psi_{\mu^2}(x,y)$ via the saddle point method tailored to the Dirichlet series 
\[\sum_{n \text{ is }y\text{-smooth}} \frac{\mu^2(n)}{n^s} =\prod_{p\le y}(1+p^{-s})= \frac{\zeta(s,y)}{\zeta(2s,y)}.\]
The relevant saddle point, denoted  $\alpha_{\mu^2}=\alpha_{\mu^2}(x,y)>0$, can be defined via 
\[ \sum_{p \le y} \frac{\log p}{p^{\alpha_{\mu^2}}+1}=\log x.\]
Let $\phi_{\mu^2}(s,y)= \log (\zeta(s,y)/\zeta(2s,y))$ and $\sigma_{2,\mu^2}=\phi_{\mu^2}^{(2)}(\alpha_{\mu^2},y)$ where the derivative is taken with respect to the first variable. Then Theorem 1 of \cite{Naimi} says that 
\begin{equation}\label{eq:naimi} \Psi_{\mu^2}(x,y) = \frac{\zeta(\alpha_{\mu^2},y)x^{\alpha_{\mu^2}}}{\zeta(2\alpha_{\mu^2},y)\alpha_{\mu^2}\sqrt{2 \pi \sigma_{2,\mu^2}}} (1+O_{\varepsilon}(1/u))
\end{equation}
holds in  $ x\ge y \ge (\log x)^{1+\varepsilon}$. In \cite{TenenbaumSfree}, La Bret\`eche and Tenenbaum pushed Naimi's analysis down to $\vartheta(y) > 2 \log x$. Their results can be applied in the complementary range  $2\log x\ge \vartheta(y)>\log x$ using the relation $\Psi_{\mu^2}(x,y) + \Psi_{\mu^2}(\prod_{p \le y}p/x,y) \in \{ 2^{\pi(y)}, 2^{\pi(y)}+1\}$. Theorem 2.1 of \cite{TenenbaumSfree} says
\begin{equation}\label{eq:TenenbaumSfree} \Psi_{\mu^2}(x,y) = \frac{\zeta(\alpha_{\mu^2},y)x^{\alpha_{\mu^2}} }{\zeta(2\alpha_{\mu^2},y)}G(\alpha_{\mu^2}\sqrt{\sigma_{2,\mu^2}})(1+O(1/u)) 
\end{equation}
holds for $\vartheta(y)>2\log x$, where
\[G(r):=e^{r^2/2}\int_{r}^{\infty}\frac{e^{-t^2/2}\diff{t}}{\sqrt{2\pi}}.\]
For $r \gg 1$,
\begin{equation}\label{eq:Gasymp}
G(r) = \frac{1-r^{-2}+O(r^{-4})}{\sqrt{2\pi}r}.
\end{equation}  
\begin{prop}\label{prop:sfree}
	Suppose $\log^3 x\ge y \ge 3\log x$. For $t>0$, let $f(t):=(t-1)/\log t$  if $t \neq 1$, and $f(1):=1$.	Then 
	\[ \frac{\Psi_{\mu^2}(x,y)}{\Psi(x,y)} = \frac{1}{\zeta(2\alpha,y)}\exp(-A_1+B_1)=\frac{1}{\zeta(2\alpha_{\mu^2},y)}\exp(A_2+B_2)\]
	where $A_i>0$, $B_i \in \mathbb{R}$ are quantities satisfying \begin{equation}\label{eq:ACsize}
		A_i \asymp  2\frac{f^2\left( \frac{(u\log(u+1))^2}{y}\right)}{u}, \qquad B_i \ll \frac{1}{u}+\frac{f\left( \frac{(u\log(u+1))^2}{y}\right)}{u}\quad (i=1,2).
	\end{equation}
	If both $y/\log x \to \infty$ and  $y \le \log^{2+o(1)}x$ hold, then we may replace the  ``$\asymp$'' sign in \eqref{eq:ACsize} with ``$\sim$''.
\end{prop}
In particular, Proposition \ref{prop:sfree} implies Corollary~\ref{cor:sfreerange}. The proposition improves \cite[Cor.~2.6]{DLBT20052} and extends \cite[Thm.~2]{Naimi} in the considered range. To prove it, we need some lemmas. 
\begin{lem}\label{lem:phi}
	Suppose $\log^3 x \ge y \ge 3 \log x$ and $x \ge C$. For all $t$ between $\alpha$ and $\alpha_{\mu^2}$, we have 
	\[\phi^{(2)}(t,y) \asymp \phi^{(2)}_{\mu^2}(t,y) \asymp \log x \log y,\]
	where the derivative is taken with respect to the first variable. If furthermore $y/\log x \to \infty$, then we may replace the  ``$\asymp$'' signs with ``$\sim$''.
\end{lem}
\begin{proof}
	Recall $\phi(t,y)= \log \zeta(t,y)$ and $\phi_{\mu^2}(t,y) = \log (\zeta(t,y)/\zeta(2t,y))$.
	We have
	\[ \phi^{(2)}(t,y) = \sum_{p \le y} \frac{p^t\log^2p}{(p^t-1)^2}, 	\qquad \phi^{(2)}_{\mu^2}(t,y) = \sum_{p \le y} \frac{p^t \log^2 p}{(p^t+1)^2}.\]
	In particular, $\phi^{(2)}(t,y)$ and $\phi_{\mu^2}^{(2)}(t,y)$ are decreasing for $t>0$, so it suffices to study them at $t=\alpha,\alpha_{\mu^2}$. By \cite[Thm.~2]{HildebrandTenenbaum1986} and \cite[Lem.~2.9]{TenenbaumSfree}, $\phi^{(2)}(\alpha,y) \asymp \phi^{(2)}_{\mu^2}(\alpha_{\mu^2},y)\asymp \log x \log y$, and if $y/\log x \to \infty$, then ``$\asymp$'' can be replaced with ``$\sim$''. It remains to understand $\phi^{(2)}_{\mu^2}(\alpha,y)$ and $\phi^{(2)}(\alpha_{\mu^2},y)$.
	
	Lemma 13 of \cite{HildebrandTenenbaum1986} estimates $\phi^{(2)}(t,y)$, which together with the estimates for $\alpha_{\mu^2}$ given in \cite[Lem.~2.8]{TenenbaumSfree} implies that  $\phi^{(2)}(\alpha_{\mu^2},y) \asymp \log x \log y$, and that if $y/\log x \to \infty$, then ``$\asymp$'' can be replaced with ``$\sim$''. It remains to estimate $\phi^{(2)}_{\mu^2}(\alpha,y)$. If $y\asymp \log x$, then $\alpha \asymp 1/\log y$ \cite[Lem.~2.8]{TenenbaumSfree} and $\phi_{\mu^2}^{(2)}(\alpha,y) \asymp \sum_{p \le y} \log^2 p \asymp y \log y \asymp \log x \log y$. If $y/\log x \to \infty$, we use the identity $\phi_{\mu^2}^{(2)}(\alpha,y)=\phi^{(2)}(\alpha,y)-4\phi^{(2)}(2\alpha,y)$ together with the estimates for $\phi^{(2)}(t,y)$  and $\alpha$ in \cite{HildebrandTenenbaum1986} to deduce $\phi^{(2)}(2\alpha,y)=o(\log x \log y)$ and so $\phi_{\mu^2}^{(2)}(\alpha,y)\sim \log x \log y$.
\end{proof}	

\begin{lem}\label{lem:alphadiff}
	Suppose $\log^3 x \ge y \ge 3\log x$ and $x \ge C$. Then 
	\[\alpha-\alpha_{\mu^2} \asymp \frac{2}{\log x \log y} \sum_{p \le y} \frac{\log p}{p^{2\alpha_{\mu^2}}-1}.\]
	If furthermore $y/\log x \to \infty$, then we may replace the  ``$\asymp$'' sign with ``$\sim$''.
\end{lem}
\begin{proof}
	Comparing the relations $\sum_{p \le y} \log p/(p^{\alpha}-1)=\log x$ and $\sum_{p \le y} \log p/(p^{\alpha_{\mu^2}}+1)=\log x$, we find
	\[\sum_{p \le y} \frac{\log p}{p^{\alpha_{\mu^2}}-1}- \sum_{p \le y} \frac{\log p}{p^{\alpha}-1}  = 2 \sum_{p \le y} \frac{\log p}{p^{2\alpha_{\mu^2}}-1}.\]
	By the mean value theorem, there is a real number $t_{x,y}$ between $\alpha$ and $\alpha_{\mu^2}$ such that
	\[ (\alpha-\alpha_{\mu^2})\phi^{(2)}(t_{x,y},y)  =2 \sum_{p \le y} \frac{\log p}{p^{2\alpha_{\mu^2}}-1}.\]
	We conclude by Lemma \ref{lem:phi}.
\end{proof}

\begin{lem}\label{lem:psum2}
	Suppose $\log^3 x \ge y \ge 3\log x$ and $x \ge C$. Then, in the notation of Proposition \ref{prop:sfree},
	\begin{equation}\label{eq:psum2}
		\sum_{p \le y} \frac{\log p}{p^{2\alpha_{\mu^2}}-1} \asymp f\left( \frac{ (u\log (u+1))^2}{y}\right)\log y. 
	\end{equation}
	If furthermore $y/\log x \to \infty$ and $y\le \log^{2+o(1)}x$, then we may replace  ``$\asymp$''  with ``$\sim$''.
\end{lem}
\begin{proof}
	If $y \asymp \log x$, then $\alpha_{\mu^2} \asymp 1/\log y$ \cite[Lem.~2.8]{TenenbaumSfree}, implying that the left-hand side of \eqref{eq:psum2} is 
	\[  \asymp \sum_{p \le y} \frac{1}{\alpha_{\mu^2}} \asymp y\]
	as needed. If $y/\log x \to \infty$, then $\alpha_{\mu^2}\log y \to \infty$ \cite[Lem.~2.8]{TenenbaumSfree}. Equation (7.1) of \cite{HildebrandTenenbaum1986} with $\sigma=2\alpha_{\mu^2}$ says that the left-hand side of \eqref{eq:psum2} equals
	\begin{equation}\label{eq:logpint}
		(1+o(1))\int_{1}^{y} \frac{\diff{t}}{t^{2\alpha_{\mu^2}}}+ O(1).
	\end{equation}
	We have the trivial lower bound 
	\[		\sum_{p \le y} \frac{\log p}{p^{2\alpha_{\mu^2}}-1} \ge \frac{\log 2}{2^{2\alpha_{\mu^2}}-1} \gg \frac{1}{\alpha_{\mu^2}} \]
	which complements \eqref{eq:logpint} when the integral is bounded.
	By \cite[Lem.~2.8]{TenenbaumSfree}, we have $y^{1-\alpha_{\mu^2}}\asymp u\log(u+1)$ in the considered range, which means that $y^{1-2\alpha_{\mu^2}} \asymp (u\log (u+1))^2/y$. If $y/\log x \to \infty$, then the same lemma shows $y^{1-2\alpha_{\mu^2}} \sim (u\log (u+1))^2/y$. This allows us to estimate the integral in \eqref{eq:logpint} as the right-hand side of \eqref{eq:psum2}. 
\end{proof}
\begin{lem}\label{lem:phidiff}
	Suppose $\log^3 x \ge y \ge 3\log x$ and $x \ge C$. Then, in the notation of Proposition \ref{prop:sfree}, \[\phi^{(2)}(\alpha,y)-\phi^{(2)}_{\mu^2}(\alpha_{\mu^2},y) \ll  f((u \log(u+1))^2/y) \log ^2 y,\]
	where the derivative is taken with respect to the first variable.
\end{lem}
\begin{proof}
	By the mean value theorem, there is a $t_{x,y} \in [\alpha_{\mu^2},\alpha]$ such that
\[	\phi^{(2)}(\alpha,y)-\phi^{(2)}_{\mu^2}(\alpha_{\mu^2},y) = \phi^{(2)}(\alpha,y)-\phi^{(2)}(\alpha_{\mu^2},y) + 4\phi^{(2)}(2\alpha_{\mu^2},y)= (\alpha-\alpha_{\mu^2})\phi^{(3)}(t_{x,y},y) +4\phi^{(2)}(2\alpha_{\mu^2},y). \]
	We appeal to Lemmas \ref{lem:alphadiff} and \ref{lem:psum2} to estimate $\alpha-\alpha_{\mu^2}$. At this point, it suffices to demonstrate $\phi^{(3)}(t_{x,y},y) \ll \log x (\log^2 y)$ and $\phi^{(2)}(2\alpha_{\mu^2},y) \ll f((u\log(u+1))^2/y)\log^2 y$. To show this, we use the elementary estimate 
	\begin{equation}\label{eq:lem4} \phi^{(k)}(t,y) \ll \log^{k-1} y \sum_{p \le y}\frac{\log p}{p^t-1}, \qquad (k=2,3)
	\end{equation}
	which 	holds as long as $t \gg 1/\log y$; it is demonstrated in the proof of \cite[Lem.~4]{HildebrandTenenbaum1986}. We may take any $t \ge \alpha_{\mu^2}$ since $\alpha_{\mu^2} \gg  1/\log y$ \cite[Lem.~2.8]{TenenbaumSfree}. From the definition of $\alpha_{\mu^2}$ and Lemma \ref{lem:psum2},
	\begin{multline*}
	\sum_{p \le y}\frac{\log p}{p^{t_{x,y}}-1} \le \sum_{p \le y}\frac{\log p}{p^{\alpha_{\mu^2}}-1} = \sum_{p \le y}\frac{\log p}{p^{\alpha_{\mu^2}}+1}+2\sum_{p \le y}\frac{\log p}{p^{2\alpha_{\mu^2}}-1}\\
	= \log x + O\bigg(f \bigg(\frac{(u\log(u+1))^2}{y}\bigg)\log y\bigg)\ll \log x,
\end{multline*}
	implying $\phi^{(3)}(t_{x,y},y) \ll \log x (\log^2 y)$ via \eqref{eq:lem4} with $k=3$. Similarly, by \eqref{eq:lem4} with $k=2$,
	\[\phi^{(2)}(2\alpha_{\mu^2},y) \ll \log y \sum_{p \le y} \frac{\log p}{p^{2\alpha_{\mu^2}}-1} \ll f\bigg( \frac{(u\log(u+1))^2}{y}\bigg) \log^2 y\]
	by Lemma \ref{lem:psum2}.
\end{proof}

\begin{proof}[Proof of Proposition \ref{prop:sfree}]
	Recall $\phi(t,y)= \log \zeta(t,y)$, $\phi_{\mu^2}(t,y) = \log (\zeta(t,y)/\zeta(2t,y))$. For $t>0$, let \[g(t) := \phi(t,y)+t \log x \quad \text{and}\quad g_{\mu^2}(t) := \phi_{\mu^2}(t,y)+t\log x.\]
	Recall Hildebrand and Tenenbaum proved in \cite[Thm.~1]{HildebrandTenenbaum1986} that
	\begin{equation}\label{eq:psiacc}
		\Psi(x,y) = \frac{\zeta(\alpha,y)x^{\alpha}}{\alpha\sqrt{2\pi \phi_2(\alpha,y)}}(1+O(1/\ubar))
	\end{equation}
	holds uniformly for $x \ge y \ge 2$,  where $\phi_2(\alpha,y)$ is defined in \eqref{eq:phi_2} and estimated in \eqref{eq:phi_2_est}. We denote $\phi_2(\alpha,y)$ by $\sigma_2$. Dividing \eqref{eq:TenenbaumSfree} by \eqref{eq:psiacc}, we find that
	\begin{equation}\label{eq:div1}
		\begin{split}
			\frac{\Psi_{\mu^2}(x,y)}{\Psi(x,y)}
			&=\frac{G(\alpha_{\mu^2}\sqrt{\sigma_{2,\mu^2}}) \sqrt{2\pi}\alpha\sqrt{\sigma_2}}{\zeta(2\alpha,y)}\exp( g_{\mu^2}(\alpha_{\mu^2})-g_{\mu^2}(\alpha))  (1+O(1/u))\\
			&=\frac{G(\alpha_{\mu^2}\sqrt{\sigma_{2,\mu^2}}) \sqrt{2\pi}\alpha\sqrt{\sigma_2}}{\zeta(2\alpha_{\mu^2},y)}\exp( g(\alpha_{\mu^2})-g(\alpha))  (1+O(1/u)).
		\end{split}
	\end{equation}
	Since $\alpha^2 \sigma_2 \gg u$ and $\alpha_{\mu^2}^2 \sigma_{2,\mu^2} \gg u$ by  \cite[Thm.~2]{HildebrandTenenbaum1986} and \cite[Lems.~2.8--2.9]{TenenbaumSfree}, we may apply \eqref{eq:Gasymp} to simplify \eqref{eq:div1} as
	\begin{equation}\label{eq:div2}
		\begin{split}
			\frac{\Psi_{\mu^2}(x,y)}{\Psi(x,y)} &=\frac{\alpha \sqrt{\sigma_2}}{\alpha_{\mu^2}\sqrt{\sigma_{2,\mu^2}}}\frac{1}{\zeta(2\alpha,y)}\exp( g_{\mu^2}(\alpha_{\mu^2})-g_{\mu^2}(\alpha))  (1+O(1/u))\\
			&=\frac{\alpha \sqrt{\sigma_2}}{\alpha_{\mu^2}\sqrt{\sigma_{2,\mu^2}}}\frac{1}{\zeta(2\alpha_{\mu^2},y)}\exp( g(\alpha_{\mu^2})-g(\alpha))  (1+O(1/u)).
		\end{split}
	\end{equation}
	We define $A_1=g_{\mu^2}(\alpha)-g_{\mu^2}(\alpha_{\mu^2})$ and  $A_2=g(\alpha_{\mu^2})-g(\alpha)$. This determines $B_1$ and $B_2$ uniquely. By construction, $g'(\alpha)=0$ and $g'_{\mu^2}(\alpha_{\mu^2})=0$. 
	Taylor-expanding $g_{\mu^2}$ about $\alpha_{\mu^2}$, we find $A_1 = (\alpha-\alpha_{\mu^2})^2 g_{\mu^2}''(t_{x,y,\mu^2})/2$ for some $t_{x,y,\mu^2}$ between $\alpha$ and $\alpha_{\mu^2}$. Similarly,  $A_2 = (\alpha_{\mu^2}-\alpha)^2 g''(t_{x,y})/2$ for some $t_{x,y}$ between $\alpha$ and $\alpha_{\mu^2}$. The estimates for $A_i$ now follow from Lemmas \ref{lem:phi}--\ref{lem:phidiff}.
	
	By \cite[Thm.~2]{HildebrandTenenbaum1986} and \cite[Lems.~2.8--2.9]{TenenbaumSfree}, we have $\alpha \asymp \alpha_{\mu^2} $ and $\sigma_{2} \asymp  \sigma_{2,\mu^2} $ in our range, implying $B_i \ll 1$, which suffices if $y \asymp \log x$. Moreover, Lemmas \ref{lem:phi}--\ref{lem:phidiff} show \[\frac{\sigma_2 - \sigma_{2,\mu^2}}{\sigma_2},\frac{\alpha-\alpha_{\mu^2}}{\alpha} \ll \frac{f\left( \frac{(u\log(u+1))^2}{y}\right)}{u}\]	
	holds in our range,	implying the needed bounds on $B_i$ once $y \ge C \log x$.
\end{proof}
\begin{rem}\label{rem:sfree}
	The term $1/u$ can be removed from the bound on $B_i$ in \eqref{eq:ACsize} as we sketch below. If $y \le \log^2 x$, then $f((u\log(u+1))^2/y) \gg 1$, so it suffices to explain how $1/u$ can be omitted when $y \in [\log^2 x,\log^3x]$. The Main Theorem of \cite{Saha} says that the $O(1/\ubar)$ term in \eqref{eq:psiacc} can be replaced with
	\begin{equation}\label{eq:uimp} \frac{1}{8} \frac{\phi^{(4)}(\alpha,y)}{\phi^{(2)}(\alpha,y)^2} - \frac{5}{24} \frac{\phi^{(3)}(\alpha,y)^2}{\phi^{(2)}(\alpha,y)^3} + O\left(\frac{1}{\log x}\right)
	\end{equation}
	when $y \in [\log^2 x,\log^3x]$. An adaptation of \cite{Saha} should show that (in the same range) the term $O_{\varepsilon}(1/u)$ in \eqref{eq:naimi}
	can be replaced with 
	\begin{equation}\label{eq:uimp2} \frac{1}{8} \frac{\phi_{\mu^2}^{(4)}(\alpha_{\mu^2},y)}{\phi_{\mu^2}^{(2)}(\alpha_{\mu^2},y)^2} - \frac{5}{24} \frac{\phi_{\mu^2}^{(3)}(\alpha_{\mu^2},y)^2}{\phi_{\mu^2}^{(2)}(\alpha_{\mu^2},y)^3} + O\left(\frac{1}{\log x}\right).
	\end{equation}
	Using an argument similar to Lemma \ref{lem:phidiff}, the difference between \eqref{eq:uimp} and \eqref{eq:uimp2} can be shown to be $\ll 1/\log x + f((u\log(u+1))^2/y)/u^2$. This upgrades the terms $1/u$ in \eqref{eq:div1}--\eqref{eq:div2} to $1/\log x+ f((u\log(u+1))^2/y)/u^2$, which in turn upgrades the bound on $B_i$ in Proposition \ref{prop:sfree} to 
	\[ B_i \ll \frac{1}{\log x} + \frac{f\left(\frac{(u\log(u+1))^2}{y}\right)}{u^2} +  \frac{f\left(\frac{(u\log(u+1))^2}{y}\right)}{u} \ll   \frac{f\left(\frac{(u\log(u+1))^2}{y}\right)}{u}.\]
\end{rem}

\section{Proofs of technical lemmas}\label{sec:technical}
\subsection{Proof of Lemma \ref{lem:ularge}}
The following lemma is a minor generalization of Lemma~10 of \cite{HildebrandTenenbaum1986}. \begin{lem}\cite[Lem.~10]{HildebrandTenenbaum1986}\label{lem:trunc}
	Fix $\varepsilon >0$. Suppose that $0<d\le x$. Uniformly for $x\ge y \ge 2$, we have
	\[\Psi(x/d,y)=\frac{1}{2\pi i}\int_{\alpha-i/\log y}^{\alpha+i/\log y} \zeta(s,y)\frac{(x/d)^s}{s}\diff{s}+O_{\varepsilon} ( (x/d)^{\alpha} \zeta(\alpha,y) \exp(-\min\{ cu\log(2u)^{-2},(\log y)^{3/2-\varepsilon}\})).\]
\end{lem}
Lemma~10 of \cite{HildebrandTenenbaum1986} corresponds to the $d=1$ case of Lemma~\ref{lem:trunc}. The proof of Lemma~\ref{lem:trunc} is in fact identical to the proof of \cite{HildebrandTenenbaum1986} if one replaces instances of ``$x$'' with ``$x/d$''. From Lemma~\ref{lem:trunc} applied twice (once as is, and a second time with $d=1$), we find that
\begin{equation}\label{eq:psiminuspsi}
	\Psi(x/d,y) - d^{-\alpha}\Psi(x,y) = \frac{1}{2\pi i} \int_{\alpha-i/\log y}^{\alpha+i/\log y} \zeta(s,y) ( d^{-s} -d^{-\alpha}) \frac{x^s}{s}\diff{s} + E
\end{equation}
for
\[ E \ll_{\varepsilon}  (x/d)^{\alpha} \zeta(\alpha,y) \exp(-\min\{ cu\log(2u)^{-2},(\log y)^{3/2-\varepsilon}\}). \]
From \eqref{eq:psiacc} and \eqref{eq:phi_2_est},
\begin{equation}\label{eq:sad}
	\Psi(x,y) \asymp \frac{\zeta(\alpha,y)x^{\alpha}}{\alpha\sqrt{(\log x)( \log y)\max\{1,(\log x)/y\}}},
\end{equation}
so we may bound $E$ by
\[ E\ll_{\varepsilon}  d^{-\alpha}\Psi(x,y)\alpha  \sqrt{(\log x) (\log y )\max\{1,(\log x)/y\}} \exp(-\min\{ cu\log(2u)^{-2},(\log y)^{3/2-\varepsilon}\}).\]
We let $T_0:=\ubar^{2/3}/\log x$, as in the proof of \cite[Lem.~11]{HildebrandTenenbaum1986}.
We shall now bound the contribution of $T_0 \le |t| \le 1/\log y$ to the integral in the right-hand side of \eqref{eq:psiminuspsi}, where $t=\Im s$. The trivial bound
\[ d^{-\alpha-it}-d^{-\alpha} \ll d^{-\alpha} \]
shows that this range contributes to the integral in the right-hand side of \eqref{eq:psiminuspsi} at most
\begin{equation}\label{eq:abovet0}
	\ll \frac{\zeta(\alpha,y) x^{\alpha}}{d^{\alpha}\alpha}   \int_{T_0}^{1/\log y} \frac{\alpha}{t+\alpha} \left|\frac{\zeta(\alpha+it,y)}{\zeta(\alpha,y)}\right| \diff{t}
\end{equation}
Using \eqref{eq:sad}, we can bound \eqref{eq:abovet0} as
\[ \ll \frac{\Psi(x,y)}{d^{\alpha}} \sqrt{(\log x )(\log y)\max\{1,(\log x)/y\}}\int_{T_0}^{1/\log y}  \left|\frac{\zeta(\alpha+it,y)}{\zeta(\alpha,y)}\right| \diff{t}.\]
If $y \ge \log x$, then
\[ \left|\frac{\zeta(\alpha+it,y)}{\zeta(\alpha,y)}\right| \le \exp(-ct^2 (\log x)( \log y))\]
by the first part of \cite[Lem.~8]{HildebrandTenenbaum1986}, and it follows that 
\[ \int_{T_0}^{1/\log y}  \left|\frac{\zeta(\alpha+it,y)}{\zeta(\alpha,y)}\right| \diff{t}\ll \exp(-cu^{1/3}) .\]
If $y \le \log x$ (and $y$ is sufficiently large\footnote{If $y =O(1)$, Lemma~\ref{lem:ularge} is easily verified since $\alpha \asymp 1/\log x$, $\additive \asymp 1$, and $\Psi(x,y)\asymp (1+\log x)^{\pi(y)}$, where $\pi(y)$ is the prime counting function.}), then the first part of \cite[Lem.~8]{HildebrandTenenbaum1986} shows
\begin{align*}
	\int_{T_0}^{1/\log y}  \left|\frac{\zeta(\alpha+it,y)}{\zeta(\alpha,y)}\right| \diff{t} &\le \int_{T_0}^{1/\log y} \left(1+ \frac{ct^2\log^2 x \log^2 y}{y^2}\right)^{-\frac{cy}{\log y}}\diff{t} = \frac{1}{\log x}\int_{\ubar^{2/3}}^{u}\left(1+ \frac{cv^2 \log^2 y}{y^2}\right)^{-\frac{cy}{\log y}}\diff{v} \\
	& \le \frac{1}{\log x} \left( \int_{\ubar^{2/3}}^{Cy/\log y} \exp\left(-\frac{cv^2  \log y}{y}\right)\diff{v} + \int_{Cy/\log y}^{u}\left(1+ \frac{cv^2 \log^2 y}{y^2}\right)^{-\frac{cy}{\log y}}\diff{v}\right)\\
	& \ll  \frac{\exp(-c\ubar^{1/3})}{\log x} \ll \frac{\exp(-c\ubar^{1/3})}{u}.
\end{align*}
At this stage, we have shown that
\[\Psi(x/d,y) - d^{-\alpha}\Psi(x,y) (1+E')= S\]
for 
\[ S:= \frac{1}{2\pi i}\int_{\alpha-iT_0}^{\alpha+iT_0} \zeta(s,y) x^s (d^{-s} - d^{-\alpha}) \frac{\diff{s}}{s}\]
and $E' \ll \additive$.
We study $S$. Using the first part of \cite[Lem.~8]{HildebrandTenenbaum1986} again gives the bound
\[ S \ll \frac{\zeta(\alpha,y)x^{\alpha}}{d^{\alpha}\alpha} \int_{\alpha-iT_0}^{\alpha+iT_0} \left| \frac{\zeta(\alpha+it,y)}{\zeta(\alpha,y)}\right|\diff{t} \ll  \frac{\zeta(\alpha,y)x^{\alpha}}{d^{\alpha}\alpha} \int_{-T_0}^{T_0} \exp(-c\phi_2(\alpha,y)t^2)\diff{t} \ll  \frac{\zeta(\alpha,y)x^{\alpha}}{d^{\alpha}\alpha\sqrt{\phi_2(\alpha,y)}}\ll \frac{\Psi(x,y)}{d^{\alpha}}\]
where the last inequality follows from \eqref{eq:psiacc}. This shows $\Psi(x/d,y) \ll_{\varepsilon} \Psi(x,y)d^{-\alpha}(1+\additive)$, proving the last part of Lemma \ref{lem:ularge}.

We now give another estimate for $S$.
Put $\phi(s,y)=\log \zeta(s,y)$ and $\sigma_k=\phi^{(k)}(\alpha,y)$, the derivative being taken with respect to the first variable. Recall that $\sigma_2$ coincides with $\phi_2(\alpha,y)$ from \eqref{eq:phi_2} and observe $T_0 \ll \alpha$. By \cite[Lem.~4]{HildebrandTenenbaum1986},
\[\sigma_k \asymp (-u \log y)^k (\ubar)^{1-k}\]
for $k=2,3,4$. As in \cite[p.~280]{HildebrandTenenbaum1986}, we Taylor-expand to find
\begin{equation}\label{eq:tay} \frac{\zeta(s,y)x^{s}}{s} = \frac{\zeta(\alpha,y)x^{\alpha}}{\alpha} e^{-\frac{t^2}{2}\sigma_2 } \left(1-i\frac{t}{\alpha}-i\frac{t^3}{6}\sigma_3 + O(t^6 \sigma_3^2 +t^2 \alpha^{-2} +t^4 \sigma_4 )\right).
\end{equation}
We Taylor-expand $d^{-s}-d^{-\alpha}$ too, as
\begin{equation}\label{eq:moretay} d^{-s}-d^{-\alpha}=d^{-\alpha}(-it\log d  +O(t^2\log^2 d))
\end{equation}
for $s=\alpha+it$ where $t= \Im s$.
Plugging \eqref{eq:moretay} and \eqref{eq:tay} in $S$ and performing the change of variables $t^2 \sigma_2 = v^2$ gives
\[ S = d^{-\alpha}\frac{\zeta(\alpha,y)x^{\alpha}}{\alpha\sqrt{2\pi \sigma_2}} \frac{1}{\sqrt{2\pi}}\int_{-T_0\sqrt{\sigma_2}}^{T_0\sqrt{\sigma_2}} e^{-\frac{v^2}{2}} S' \diff{v} \]
where
\[S' =-v^2\frac{\log d}{\alpha \sigma_2} - v^4\frac{ \log d }{6\sigma_2^2}\sigma_3+v z+ O\left( \frac{v^2 \log ^2 d}{\sigma_2} + \frac{|v| |\log d|}{\sqrt{\sigma_2}}\left(v^6 \frac{\sigma_3^2}{\sigma_2^3} + v^2 \frac{\alpha^{-2}}{\sigma_2} + v^4\frac{ \sigma_4}{\sigma_2^2}\right)\right)\]
for some quantity $z$ whose precise value does not matter because $\int_{-A}^{A} e^{-\frac{v^2}{2}}v \diff{v} = 0$. To estimate the integral of $e^{-v^2/2}S'$, we use 
\begin{equation}\label{eq:intest}
	\frac{1}{\sqrt{2\pi}} \int_{-A}^{A} e^{-\frac{v^2}{2}}v^k \diff{v} = \frac{k!}{(k/2)!2^{k/2}}+ O_k(e^{-c_k A^2})
\end{equation}
for even $k$ and $A \gg 1$ and 
\[	\int_{-A}^{A} e^{-\frac{v^2}{2}}|v|^k \diff{v} \ll_k 1  \]
for all $k \ge 0$. Hence,
\[	S = d^{-\alpha} \frac{\zeta(\alpha,y)x^{\alpha}}{\alpha\sqrt{2\pi \sigma_2}} \bigg( -\left(\frac{\log d}{\alpha \sigma_2}+ \frac{\log d }{2\sigma_2^2}\sigma_3\right) (1+O(\exp(-c\ubar^{1/3}))) + O\left( \frac{\log^2 d}{\sigma_2}+ \frac{|\log d|}{\sqrt{\sigma_2}} \big( \ubar^{-1} +\frac{ \alpha^{-2}}{\sigma_2}\big)\right)\bigg). \]
By \eqref{eq:psiacc}, 
\[	S= d^{-\alpha} \Psi(x,y) \left( -\left(\frac{\log d}{\alpha \sigma_2}+ \frac{\log d }{2\sigma_2^2}\sigma_3\right) (1+O(\ubar^{-1} )) + O\left( \frac{\log^2 d}{\sigma_2}+\frac{|\log d|}{\sqrt{\sigma_2}} \big( \ubar^{-1} + \frac{\alpha^{-2}}{\sigma_2}\big)\right)\right). \]
In summary,
\begin{equation}\label{eq:psixdye''}
	\Psi(x/d,y) = d^{-\alpha}\Psi(x,y) \left(1-\log d\left(\frac{1}{\alpha \sigma_2}+ \frac{\sigma_3}{2\sigma_2^2}\right)(1+O(\ubar^{-1})) +E''\right)
\end{equation}
for
\[	E'' \ll \additive+  \frac{\log^2 d}{\sigma_2} + \frac{|\log d|}{\sqrt{\sigma_2}}\big( \ubar^{-1}+\frac{\alpha^{-2}}{\sigma_2}\big)\ll\additive+  \frac{\log^2 d}{\sigma_2} + \frac{|\log d|}{\ubar\sqrt{\sigma_2}}.\]
All the terms that are multiples of $\log d$ (but not of $\log^2 d$), either in the bound for $E''$ or in \eqref{eq:psixdye''}, contribute $\ll |\log d| / \log x$, which finishes the proof if $y \ge \log x$.
We now assume $y\le \log x$. We have $\alpha \ll y/(\log x \log y)$ and in particular $\alpha \log p= O(1)$ for $p \le y$. It follows that, using the formula for $\sigma_2=\phi_2(\alpha,y)$ given in \eqref{eq:phi_2},
\[ \sigma_2 =  \sum_{p \le y} \frac{(\log p)^2(1+O(\alpha \log p))}{(\alpha \log p +O(\alpha^2 \log^2p))^2} = \frac{\sum_{p \le y} (1+O(\alpha \log p))}{\alpha^2} = \frac{\pi(y) + O(\alpha y)}{\alpha^2} = \pi(y)\frac{1+ O(y/\log x)}{\alpha^2}\]
where we used $\pi(y)\asymp y/\log y$. Similarly, using the formula for $\sigma_3$ given in \cite[Rem.~1]{Saha},
\[ -\sigma_3 = \sum_{p \le y} \frac{(\log p)^3 (p^{\alpha}+p^{2\alpha})}{(p^{\alpha}-1)^3} =2\pi(y)\frac{1+ O(y/\log x)}{\alpha^3}.\]
Hence,
\[ \frac{1}{\alpha\sigma_2} + \frac{\sigma_3}{2\sigma_2^2} = \frac{1}{\sigma_2} ( \alpha^{-1} -\alpha^{-1}(1+O(y/\log x)) ) \ll \frac{y}{(\log x)^2}\]
which allows us to bound the terms in \eqref{eq:psixdye''} except for $E''$. To conclude the proof, we notice that if we use more terms in \eqref{eq:tay}, namely, expand the integrand as
\begin{multline*}
	\frac{\zeta(s,y)x^{s}}{s} = \frac{\zeta(\alpha,y)x^{\alpha}}{\alpha} e^{-\frac{t^2}{2}\sigma_2 }\\
	\cdot \big( \big(1-i\frac{t}{\alpha}-\frac{t^2}{\alpha^2}\big)\big(1-i\frac{t^3}{6}\sigma_3+\frac{t^4}{24}\sigma_4 - \frac{t^6}{2 \cdot 6^2}\sigma_3^2 + \frac{t^8}{2 \cdot 24^2} \sigma_4^2 - i\frac{t^7}{6\cdot 24}\sigma_3 \sigma_4 \big)\\
	+O(|t|^5 |\sigma_5| +|t|^9 |\sigma_3|^3 + t^{12}\sigma_4^3 + |t|^3 \alpha^{-3}  )\big),
\end{multline*}
then our bound on $E''$ improves to
\[ 	E'' \ll \additive+  \frac{\log^2 d}{\sigma_2} + \frac{|\log d| \log y}{y \log x }.\]
\subsection{Proof of Lemma \ref{lem:smallu}}
It is convenient to consider separately $y \le x^{1/4}$ and $y \ge x^{1/4}$. We first study $y \le x^{1/4}$. When $d \le x^{1/4}$, we argue as follows. By \eqref{eq:betaalpha}, it suffices to prove Lemma \ref{lem:smallu} with $\alpha$ replaced by $\beta$, where $\beta$ is defined in \eqref{eq:betadef}. By Saias \cite{Saias} we have, for $u \ge 3$ and $y \ge \exp(\sqrt{\log x} \log \log x)$,
\begin{equation}\label{eq:saias}
	\Psi(x,y) = x\rho(u) \left(1 + \frac{a_1\rho'(u)}{\rho(u)\log y} + O\left(\frac{|\rho''(u)|}{\rho(u)\log^2 y}\right)\right)
\end{equation}
for some constant $a_1$.
Since $\rho'(u)= -\rho(u-1)/u$ and $\rho(u) \asymp \rho(u-1)/(u \log u)$ for $u \ge 3$ \cite[Lem.~1]{Hildebrand1984}, this simplifies as
\begin{equation}\label{eq:psixy}
	\Psi(x,y) = x\rho(u) \left(1 + \frac{a_1\rho'(u)}{\rho(u)\log y} + O\left(\frac{\log^2 (u+1)}{\log^2 y}\right)\right).
\end{equation}
Since $d \le x^{1/4}\le x/y^3$, we may apply \eqref{eq:psixy} with $x/d$ in place of $x$. Let
\[ \delta := \frac{\log d}{\log y} \le u-3.\]
Then
\begin{equation}\label{eq:psixdy} \Psi(x/d,y) = \frac{x}{d}\rho(u-\delta) \left(1 + \frac{a_1\rho'(u-\delta)}{\rho(u-\delta)\log y} + O\left(\frac{\log^2 (u+1)}{\log^2 y}\right)\right).
\end{equation}
Let $r(t):=-\rho'(t)/\rho(t)$. Dividing \eqref{eq:psixdy} by \eqref{eq:psixy}, we find
\[ \Psi(x/d,y) = \frac{\Psi(x,y)}{d} \frac{\rho(u-\delta)}{\rho(u)} \left( 1 + \frac{a_1}{\log y} (r(u)-r(u-\delta)) + O\left( \frac{\log^2 (u+1)}{\log^2 y}\right)\right).\]
It remains to simplify the last equation.
Recall $\xi(u) \asymp \log (u+1)$ and $\xi'(u) \asymp 1/u$ for $u \ge 1$ \cite[Lem.~1]{Hildebrand1984}. In \cite[Lem.~3.7]{DLBT2005}, it was shown that $r(t)=\xi(t)+O(1/t)$ holds for $t>1$ and $r'(t)=\xi'(t)+O(1/t^2)$ holds for $t>2$. It follows that 
\[ r(u)-r(u-\delta) \ll \frac{\delta}{u}\]
by the mean value theorem. The logarithm of the ratio $\rho(u-\delta)/\rho(u)$ can be written as \[\log \rho(u-\delta)-\log \rho(u)= \int_{u-\delta}^{u} r(t)\diff{t} = \int_{u-\delta}^{u} \xi(t)\diff{t} + O(\delta/u)\]
and the integral of $\xi(t)$ is
\[ \int_{u-\delta}^{u} \xi(t)\diff{t} = \delta\xi(u) + \int_{u-\delta}^{u} (\xi(t)-\xi(u))\diff{t} = \delta\xi(u) + \int_{u-\delta}^{u} \Theta((t-u)/u )\diff{t}= \delta\xi(u) -\Theta(\delta^2)/u.\]
Here $\Theta(B)$ indicates a quantity that is in $[cB,CB]$. To conclude, observe $e^{\delta\xi(u)}/d=d^{-\beta}$ and $\exp(-\Theta(\delta^2)/u)=1+O(\min\{1,\delta^2/u\})$. We now treat $x \ge d \ge x^{1/4}$. Due to the error term $\log d /\log x$ in \eqref{eq:state2more2}, it suffices to show $\Psi(x/d,y) \ll \Psi(x,y)d^{-\alpha}$ in our range of parameters. Using $\Psi(x,y)/x \asymp \rho(u)$ for $y \ge \exp(\sqrt{\log x} \log \log x)$ \cite{db}, it suffices to show \[ \rho(u-\delta)/\rho(u)\ll d^{1-\alpha}.\]
Recall $\alpha-1 = -\xi(u)/\log y + O(1/(\log x \log y))$ by \eqref{eq:betaalpha}. The required inequality is a consequence of $\rho(u-\delta)/\rho(u)\ll e^{\delta \xi(u)}$ for $u\ge \delta \ge 0$ \cite[Cor.~2.4]{HT1993} and we are done.

It remains to consider $x \ge y \ge x^{1/4}$. We repeat the last argument but  instead of utilizing \eqref{eq:saias}, we make use of  de Bruijn's estimate \cite{db}
\begin{equation}\label{eq:psidb}
	\Psi(x,y) = x\rho(u) \left(1 + O\left(\frac{1}{\log y} + \frac{\mathbf{1}_{y > x}}{x}\right)\right),
\end{equation}
which holds uniformly for $y \ge x^{1/4}$. Applying \eqref{eq:psidb} once as is, and another time with $x/d$ in place of $x$, we obtain the desired estimate by comparing.
\begin{rem}\label{rem:smallu}
	If $y<x^{1/4}$, the argument above in fact proves Lemma \ref{lem:smallu} with the error term $1/\log x$ replaced by $\log^2(u+1)/\log^2 y$. If $y \ge x^{1/4}$, we claim we can improve the term $1/\log x$ as long as we are not in one of the following two scenarios: 1) $y \in [\sqrt{x/d},\sqrt{x}]$ and $d \sim 1$ and 2) $y \asymp x$. Precisely, let $y \in [x^{1/4},x]$. If $1 \le d \le x/y^2$ or $1 \le d \le x/y\le y$, we claim that
	\begin{equation}\label{eq:psixd}
		\Psi(x/d,y) = \frac{\Psi(x,y)}{d^{\alpha}}\left(1 + O\left( \frac{\log d}{\log x} + \frac{1}{\log^2 x} +\frac{y}{x\log x} \right)\right)
	\end{equation}
	holds. In showing \eqref{eq:psixd}, we may assume $y\le x/3$ and $1 \le d \le 2$, since otherwise it is in Lemma \ref{lem:smallu}. From de Bruijn's estimate for $\Psi(x,y)$ in terms of de Bruijn's approximation $\Lambda(x,y)$ \cite{db}, and the estimate for $\Lambda(x,y)$ given in  \cite[Prop.~A.5]{Gorodetsky}, we find that
	\begin{equation}\label{eq:PsiK} \Psi(x,y) = x \rho(u) K(-r(u)/\log y)\left(1+O\left( \frac{1}{\log^2 x} + \frac{y}{x \log x}\right)\right)
	\end{equation}
	holds uniformly for $y \in [x^{1/4},x]$ where $r(u)=-\rho'(u)/\rho(u)$ and $K(t)= (t+1)\zeta(t+1)/t$. Applying \eqref{eq:PsiK} once as is, and another time with $x/d$ in place of $x$, we claim \eqref{eq:psixd} follows, as we now explain. Let $I=[\log(x/d)/\log y, \log x /\log y]$. The condition $d\le x/y^2$ ensures $I \subset [2,\infty)$, while the condition $1 \le d\le x/y\le y$ ensures $I \subset [1,2]$. Since $r$ is differentiable in $(1,2)$ and in $(2,\infty)$, this allows us to estimate $\log K(-r(u)/\log y)-\log K(-r(\log(x/d)/\log y))$ as $O( \log d/\log y)=O(\log d/\log x)$ by the mean value theorem.
\end{rem}
\subsection{Proof of Lemma \ref{lem:psipsi}}
	Fix $\varepsilon>0$. Let $T_0:=\ubar^{2/3}/\log x$. As in the proof of Lemma \ref{lem:ularge},
	\begin{equation}\label{eq:psiperron5} \Psi(x/d,y)=\frac{1}{2\pi i}\int_{\alpha-iT_0}^{\alpha+iT_0}\zeta(s,y)\frac{(x/d)^s}{s}\diff{s}	+O_{\varepsilon} ( d^{-\alpha}\Psi(x,y) \additive)
	\end{equation}
	for $d \in \{1,a,b,ab\}$ if $ab \le x$. By \cite[Lem.~11]{HildebrandTenenbaum1986} and \eqref{eq:psiacc},
	\[ \int_{\alpha-iT_0}^{\alpha+iT_0} \zeta(s,y)\frac{(x/d)^s}{s}\diff{s}  \ll d^{-\alpha}\Psi(x,y)\]
	for $d \in \{1,a,b,ab\}$ if $ab \le x$. Applying \eqref{eq:psiperron5} with $d=1$ and $d=ab$ and multiplying, we find
	\[ \Psi(x,y)\Psi(x/(ab),y) =  \frac{(ab)^{-\alpha}}{(2\pi i)^2} \int_I\int_I  \zeta(s_1,y)\zeta(s_2,y) \frac{x^{s_1+s_2}}{s_1 s_2} (ab)^{\alpha-s_2}\diff{s_1}\diff{s_2} + O_{\varepsilon}((ab)^{-\alpha} \Psi(x,y)^2(\additive+\additive^2)) \]
	where $I:=\{\alpha+it : |t| \le T_0\}$. Applying \eqref{eq:psiperron5} with $d=a$ and $d=b$ and multiplying, we find
	\[ \Psi(x/a,y)\Psi(x/b,y) = \frac{(ab)^{-\alpha}}{(2\pi i)^2} \int_I\int_I \zeta(s_1,y)\zeta(s_2,y) \frac{x^{s_1+s_2}}{s_1 s_2} a^{\alpha-s_1}b^{\alpha-s_2}\diff{s_1}\diff{s_2} + O_{\varepsilon}( (ab)^{-\alpha}\Psi(x,y)^2(\additive+\additive^2) ).\]
	We subtract the last two relations to find
	\begin{multline}\label{eq:nonsym}
		\Psi(x,y)\Psi(x/(ab),y) - \Psi(x/a,y)\Psi(x/b,y) \\
		= (ab)^{-\alpha} \left( \frac{1}{(2\pi i)^2} \int_I \int_I \zeta(s_1,y)\zeta(s_2,y) \frac{x^{s_1+s_2}}{s_1 s_2} ((ab)^{\alpha-s_2}-a^{\alpha-s_1}b^{\alpha-s_2})\diff{s_1}\diff{s_2} + O_{\varepsilon}( \Psi(x,y)^2(\additive+\additive^2)) \right).
	\end{multline}
	We symmetrize \eqref{eq:nonsym} by exchanging the roles of $s_1$ and $s_2$ there, and adding back to the original form of \eqref{eq:nonsym}, obtaining
	\begin{multline}\label{eq:sym}
		\Psi(x,y)\Psi(x/(ab),y) - \Psi(x/a,y)\Psi(x/b,y) \\
		= \frac{(ab)^{-\alpha}}{2} \bigg( \frac{1}{(2\pi i)^2} \int_I \int_I \zeta(s_1,y)\zeta(s_2,y) \frac{x^{s_1+s_2}}{s_1 s_2} (a^{\alpha-s_1}-a^{\alpha-s_2})(b^{\alpha-s_1}-b^{\alpha-s_2})\diff{s_1}\diff{s_2} \\
		+ O_{\varepsilon}( \Psi(x,y)^2(\additive +\additive^2)) \bigg).
	\end{multline}
	We use the Taylor approximations \eqref{eq:tay} and \eqref{eq:moretay} to write the last integral as
	\begin{align*}
		\frac{1}{(2\pi i)^2} &\int_I \int_I \zeta(s_1,y)\zeta(s_2,y) \frac{x^{s_1+s_2}}{s_1 s_2} (a^{\alpha-s_1}-a^{\alpha-s_2})(b^{\alpha-s_1}-b^{\alpha-s_2})\diff{s_1}\diff{s_2} \\
		&=-\log a \log b\frac{\zeta(\alpha,y)^2 x^{2\alpha}}{(2\pi)^2\alpha^2} \left( \int_I \int_I e^{-\frac{t_1^2+t_2^2}{2} \sigma_2} (t_1-t_2)^2  \diff{t_1}\diff{t_2} + O(\sigma_2^{-2} E)\right)
	\end{align*}
	for
	\[ E =  \frac{\log a\log b}{\sigma_2} + \frac{\log (ab)}{\sigma_2^{1/2}} + \ubar^{-1}\]
	where we used the fact that $\int_{-A}^{A} e^{-v^2/2}v^k \diff{v}=0$ for odd $k$. For the remaining integral,
	\begin{align*}
		\int_I\int_I e^{-\frac{t_1^2+t_2^2}{2} \sigma_2} (t_1-t_2)^2 \diff{t_1}\diff{t_2}&=\sigma_2^{-2}  \int_{-T_0\sqrt{\sigma_2}}^{T_0\sqrt{\sigma_2}}
		\int_{-T_0\sqrt{\sigma_2}}^{T_0\sqrt{\sigma_2}}e^{-\frac{v_1^2+v_2^2}{2}} (v_1-v_2)^2  \diff{v_1}\diff{v_2}\\
		&=2\sigma_2^{-2} \int_{-T_0\sqrt{\sigma_2}}^{T_0\sqrt{\sigma_2}}e^{-\frac{v_1^2}{2}}v_1^2  \diff{v_1} \int_{-T_0\sqrt{\sigma_2}}^{T_0\sqrt{\sigma_2}}e^{-\frac{v_2^2}{2}}  \diff{v_2}\\
		&=2\sigma_2^{-2}  (   \sqrt{2\pi} + O(e^{-c\ubar^{2/3} }))^2
	\end{align*}
	where in the last equality, we used \eqref{eq:intest}. To conclude, we divide \eqref{eq:sym} by $\Psi(x,y)^2$ and appeal to \eqref{eq:psiacc}.

\bibliographystyle{abbrv}
\bibliography{references}
\end{document}